\documentclass[cs4size,a4paper,fancyhdr,12pt]{article}
\usepackage{amsfonts}
\usepackage{amssymb}
\usepackage{amscd}
\usepackage{mathrsfs}
\usepackage{enumerate}
\usepackage{amsmath,amssymb, amsthm}
\usepackage{graphicx}
\usepackage{indentfirst,latexsym,amssymb,bm}
\usepackage[bookmarks=false]{hyperref}
\usepackage{tikz}
\usepackage{circuitikz}
\usepackage{longtable}
\usepackage{array}
\usepackage{tabularx} % 提供 X 列类型（自动调整宽度）
\usepackage{booktabs}
\usepackage{pdflscape} 
\usepackage{float} 

\newtheorem{theorem}{Theorem}[section]
\newtheorem{pro}[theorem]{Proposition}
\newtheorem{lem}[theorem]{Lemma}
\newtheorem{rem}[theorem]{Remark}
\newtheorem{cor}[theorem]{Corollary}
\newtheorem{dnt}[theorem]{Definition}

\def\M{\mathcal{M}}

\def\Core{\hbox{\rm Core}}
\def\Aut{\hbox{\rm Aut}}
\def\Mon{\hbox{\rm Mon}}

\def\di{~\big | ~}

\def\a{\alpha}

 \def\O{\Omega}   \def\o{\omega}\def\G{\varGamma}
\def\ol1{\overline 1}
\def\di{\bigm|}

\numberwithin{equation}{section}

\begin{document}

\title{
Monodromy representation of graphs
}

\date{}

\author{Kai Yuan\footnote{Corresponding author.  Supported by the National Natural Science Foundation of China (No. 12101535). \newline
{\it E-mail address}: pktide@163.com (Kai Yuan), wang$_{-}$yan@pku.org.cn (Yan Wang).}, Yan Wang\\
{\small School of Mathematics and Information Science, Yantai University, } \\ {\small Yantai 264005, China}
}

\maketitle

\begin{abstract}
It is well-known that every vertex-transitive graph admits a representation as a  coset graph. In this paper, we extend this construction by introducing monodromy graphs defined through double cosets. Our main result establishes that every graph is isomorphic to a monodromy graph, providing a new combinatorial framework for graph representation. Moreover, we show that every  graph gives rise to an arc-transitive graph through its monodromy representation. Inspired by the monodromy representation of graphs, we denote an algebraic map $\mathcal{M}(G;\O,\rho,\tau)$ by $\mathcal{M}(G;U,\rho,\tau)$ where $U$ is a stabiliser in $G$. As an application, we prove an enumeration theorem for orientable maps with a given monodromy group. We underscore a fundamental triad in algebraic graph theory: Where there is a graph, there is a group, an arc-transitive graph, and an orientable regular map--each arising canonically from the underlying combinatorial and algebraic structures. 
\end{abstract}

{\small {\em 2000 Mathematics Subject Classification}:\,
              05C25, 05C30, 05E16, 05C62.

             {\em Keywords}: Graphs; Monodromy graphs; Double cosets; Graph representations; Arc-transitive graphs; Regular maps}

\section{Introduction}
Graphs considered in this paper are connected and may have loops, multiple edges, and free edges\cite{JS}. A map is an embedding of a graph on surfaces. Modern foundations of the theory of maps on orientable surfaces can be found in Jones and Singerman\cite{JS}.

We use the idea (due to Heffter and more recently attributed to Edmonds) of associating with a map $\mathcal{M}$ on an orientable surface $\mathcal{S}$ a pair of permutations $\tau, \rho$ of the set $\Omega$ of darts (directed edges) of $\mathcal{M}$: $\tau$ is
 the involution which reverses the direction of each dart and $\rho$ cyclically
 permutes the darts directed towards each vertex $v$ by following the
 orientation around $v$.

The representation of graphs through group-theoretic constructions has been a fundamental tool in algebraic graph theory. While  coset graphs provide a natural representation for vertex-transitive graphs\cite{S}, the general case for arbitrary connected simple graphs has remained less explored. This paper is also motivated by  the rotation systems in topological graph theory\cite{GT}. In this paper, we introduce the concept of monodromy graphs and demonstrate their universality for representing graphs.

In recent decades, the interplay between graph theory, group theory, and topological combinatorics has led to profound insights into the structure and symmetry of discrete objects. 
A central theme is the representation of graphs as coset graphs, which naturally captures vertex-transitive symmetry. 
However, many graphs are not vertex-transitive, and a more general framework is needed to encode their structure in group-theoretic terms. 
This paper addresses this gap by introducing monodromy graphs, a construction based on double cosets that generalizes the classical coset graph representation.

Our main contribution is to show that every graph—including those with loops, multiple edges, or free edges—is isomorphic to a monodromy graph. 
This provides a unified combinatorial framework for graph representation that extends beyond the vertex-transitive case. 
Moreover, we demonstrate that every graph gives rise to an arc-transitive graph via its monodromy representation, highlighting a natural pathway from arbitrary graphs to highly symmetric structures.

Inspired by the monodromy representation, we also revisit the algebraic theory of maps. 
We denote an algebraic map by $\mathcal{M}(G;U,\rho,\tau)$, where $U$ is a stabilizer in $G$, and establish a correspondence between maps and their monodromy graphs. 
As an application, we prove an enumeration theorem for orientable maps with a given monodromy group, illustrating the utility of our approach in combinatorial enumeration.

A key philosophical insight underpinning our work is the existence of a fundamental triad in algebraic graph theory:
\begin{quote}
\emph{Where there is a graph, there is a group, an arc-transitive graph, and an orientable regular map—each arising canonically from the underlying combinatorial and algebraic structures.}
\end{quote}
This triad not only unifies several classical concepts but also opens new avenues for studying graphs via their associated symmetric counterparts.

The paper is organized as follows. 
In Section~2, we review group actions, double cosets, and algebraic maps. 
Section~3 introduces monodromy graphs and proves their universality. 
Section~4 studies automorphisms of monodromy graphs and their relation to arc-transitive graphs. 
Section~5 presents an enumeration theorem for maps with a given monodromy group. 
Finally, in Section~6, we illustrate our theory with examples of Platonic solids, using computational support from Magma~\cite{BCP}.

\section{Group actions and maps}

For any positive integer $n$, $[n]$ denotes the set $\{1, 2, \ldots , n\}$.

\subsection{Group actions}

Let \(G\) and \(H\) be groups acting on the sets \(\Omega\) and \(\Delta\), respectively. The two actions are said to be \emph{permutationally isomorphic} if there exist a bijection \(\vartheta\colon\Omega\to\Delta\) and an isomorphism \(\chi\colon G\to H\) such that

\[\omega^{g\vartheta}=\omega^{\vartheta(g^\chi)}\quad\text{for all}\quad\omega\in \Omega,\;g\in G.\]

If such conditions hold, the pair \((\vartheta,\chi)\) is said to be a \emph{permutational isomorphism}.

\begin{pro}{\rm \cite[Lemma 2.8]{PS}}\label{iso}
Suppose that \(G\) acts on a set \(\Omega\) transitively and let \(\omega \in \Omega\). Define the map \(\gamma \colon \Omega \to [G : G_{\omega}]\) by the rule \(\omega^{\prime} \mapsto G_{\omega}g\) where \(g \in G\) is chosen such that \(\omega^g = \omega^{\prime}\). Then \(\gamma\) is a well defined bijection, and \((\gamma, \operatorname{id}_{G})\) is a permutational isomorphism.
\end{pro}

\begin{pro} {\rm \cite{BH}} \label{dbc}
  Let \( U \) and \( V \) be subgroups of \( G \). The double cosets of \( G \) with respect to \( U \) and \( V \) are the sets
\[ UgV = \{ ugv \mid u \in U, v \in V \} \quad (\text{with } g \in G). \]
If \( UgV \cap UhV \neq \emptyset \), then \( UgV = UhV \). So the double cosets give rise to a partition of \( G \). If \( G \) is finite, then  
\[ G = \bigcup_{i=1}^n Ug_iV \quad (\text{disjoint}) \]
and  
\[ |G| = \sum_{i=1}^n \frac{|U||V|}{|U^{g_i} \cap V|}, \]
where \( U^{g_i} = g_i^{-1} U g_i \). In particular,  the number of right cosets of $U$ in $G$ that lie in $Ug_iV$ is $\frac{|V|}{|U^{g_i} \cap V|}$.
 \end{pro}

\begin{lem}\label{suborb}
Let \( G \) act transitively on \(\Omega\), let \( G_\o \) be the stabilizer subgroup of $\o\in \O$ and $H$ a subgroup of $G$. Suppose
\[G = \bigcup_{i=1}^t G_\o g_i H\]
is the double coset decomposition of \( G \) with respect to the subgroups $ G_\o$ and $H$, where $g_i\in G, 1\leq i\leq t$ and \( g_1 \in G_\o \) in particularly. Then
    \(\Delta_i = \{\o^g \mid g \in G_\o g_i H\}\), for \( i = 1, \ldots, t \), are precisely all the orbits of \( H\) on \(\Omega\).

\end{lem}

\begin{proof}
  Since $\Delta_i = \o^{G_\o g_i H} = (\o^{g_i})^{H}$, it follows that $\Delta_i$ is an orbit of $H$. By the transitivity of $G$, we have $\bigcup_{i=1}^t \Delta_i = \Omega$.

It remains to show that the $\Delta_i$ are pairwise distinct. Suppose $\Delta_i = \Delta_j$, which means $\o^{G_\o g_i H} = \o^{G_\o g_j H}$. Then there exists $h \in H$ such that $\o^{g_i h} = \o^{g_j}$, which implies $g_i h g_j^{-1} \in G_\o$.
This yields $g_i \in G_\o g_j H$, and consequently $G_\o g_i H = G_\o g_j H$, which means $i = j$.
\end{proof}

\subsection{Maps}

An \emph{algebraic map} $\mathcal{M}$ is defined as a quadruple $\mathcal{M}(G; \Omega, \rho, \tau)$, 
where $\Omega$ is a set of darts, and $\tau, \rho$ are permutations of $\Omega$ such that 
$\tau^2 = 1$ and the group $G = \langle \tau, \rho \rangle$ acts transitively on $\Omega$.
The underlying graph $\Gamma$ of $\mathcal{M}$ has vertex set $V = V(\Gamma)$ and edge set 
$E = E(\Gamma)$, and $\Omega$ corresponds to the darts of $\Gamma$. Each dart $\alpha \in \Omega$ 
is associated with a unique vertex $\text{vert}(\alpha) \in V$ and a unique edge $\text{edge}(\alpha) \in E$.

For a vertex $v \in V$, the set 
$$
\text{dart}(v) = \{ \alpha \in \Omega \mid \text{vert}(\alpha) = v \}
$$ 
consists of all darts incident to $v$, and the valency $\text{val}(v) = |\text{dart}(v)|$ satisfies 
$1 \leq \text{val}(v) < \infty$. A \emph{free point}---that is, an endpoint of a free edge (which is not 
considered a vertex)---is defined to have valency 1.

For an edge $e \in E$, the set 
$$
\text{dart}(e) = \{ \alpha \in \Omega \mid \text{edge}(\alpha) = e \}
$$ 
has size 1 if $e$ is a free edge, and size 2 otherwise.

The permutation $\tau$ acts on darts as follows:
\begin{itemize}
    \item If $\text{edge}(\alpha)$ is a loop, then $\tau$ exchanges the two darts on that loop, 
          so $\alpha^\tau \neq \alpha$ but $\text{edge}(\alpha^\tau) = \text{edge}(\alpha)$.
    \item If $\text{edge}(\alpha)$ is a free edge, then $\tau$ fixes the single dart on that edge, 
          so $\alpha^\tau = \alpha$.
\end{itemize}

\begin{pro}{\rm\cite[Lemma 2.1a]{JS}}\label{2.1a}
There is a natural bijection between $E$ and the cycles of $\tau$
on $\Omega$, with darts $\alpha, \beta \in \Omega$ lying in the same cycle
if and only if $\text{edge}(\alpha) = \text{edge}(\beta)$.
\end{pro}

\begin{pro}{\rm\cite[Lemma 2.1b]{JS}} \label{2.1b}
There is a natural bijection between $V$ and the cycles of $\rho$
on $\Omega$, with darts $\alpha, \beta \in \Omega$ lying in the same cycle
if and only if $\text{vert}(\alpha) = \text{vert}(\beta)$.
\end{pro}

Suppose that $\a\in \O$ and $U=G_\a$. Clearly, $U$ is core free in $G$. From Proposition 6.10 in \cite{JS}, we can get the following proposition.

\begin{pro}\label{grpro}
\begin{enumerate}
    \item[{\rm (i)}] $\M$ has free edges if and only if $\tau \in U^g ~(= g^{-1}Ug)$,
    for some $g \in G$.

    \item[{\rm (ii)}] $\M$ has loops if and only if for some $g \in G$, $\tau\rho^i \in U^g$ but $\tau \notin U^g$.

    \item[{\rm (iii)}] $\M$ has multiple edges if and only if for some $g \in G$,
    $\tau\rho^i\tau\rho^j \in U^g$ $(i \neq 0, j \neq 0)$ where $\tau\rho^i\tau \notin U^g$,
    and for all $k \in \mathbb{Z}$, $\tau\rho^k \notin U^g$.
\end{enumerate}
\end{pro}

We immediately deduce the following proposition from Theorem 6.11 in \cite{JS} or Proposition \ref{grpro}.
\begin{pro} \label{simple}
$\G$ is a simple graph if and only if

\begin{enumerate}
    \item[{\rm (i)}] For all $i \in \mathbb{Z}$ and for all $g \in G$, $\tau\rho^i \notin U^g$,

    \item[{\rm (ii)}] Whenever $\tau\rho^i\tau\rho^j \in U^g$ for some $i,j \in \mathbb{Z}$ then $\tau\rho^i\tau \in U^g$
    $($equivalently $\rho^j \in U^g)$.
\end{enumerate}
\end{pro}

\section{Monodromy graphs}

In this section, we introduce monodromy graphs and show that every connected simple graph is isomorphic to a monodromy graph.

\begin{dnt}[Multiplicity of an Edge]
Let \( \G \) be a graph. The \textbf{multiplicity} of an edge between vertices \( u \) and \( v \) is denoted by \( m(u, v) \), defined as:
\[
m(u, v) = \text{Number of edges connecting } u \text{ and } v.
\]
For loops (\( u = v \)), \( m(u, u) \) counts the number of self-loops at \( u \). For free edge, \( m(u, u) \) counts the number of free edges at \( u \).
\end{dnt}

\begin{dnt} \label{mong}
Suppose that $G=\langle\rho, \tau\rangle$ with $\tau^2=1$ is a group and $U$ is a core free subgroup of $G$. Let $S$ be a set of right
coset representatives $($or a right transversal$)$ for $U$ in $G$. The graph \( \G = \Mon(G; U, \rho, \tau) \) is called a  monodromy graph with vertex set $$V=\{Ug\langle \rho \rangle \mid g \in G\},$$
edge set $$E=\{\{Uh\langle \rho \rangle, Uh\tau \langle \rho \rangle\} \mid h\in S\},$$
and for $h\in S$, $$m(Uh\langle \rho \rangle, Uh\tau \langle \rho \rangle)=|D|$$
where $D=\{g\in S \mid g\in Uh\langle \rho \rangle \hbox{~and~} g\tau \subseteq Uh\tau \langle \rho \rangle\}$.
\end{dnt}

\begin{rem}
  An edge $\{Uh\langle \rho \rangle, Uh\tau \langle \rho \rangle\}$ is called a free edge if $h\tau\in Uh$. An edge $\{Uh\langle \rho \rangle, Uh\tau \langle \rho \rangle\}$ is called a loop if $h\tau\not\in Uh$ and $h\tau\in Uh\langle \rho \rangle$.
\end{rem}

\begin{lem} \label{epro}
Let \( \G = \Mon(G; U, \rho, \tau) \) as defined in Definition {\rm \ref{mong}}. Let $T$ be a set of right
coset representatives for $U$ in $G$. Then the edge set
$$E(\G)=\{\{Ut\langle \rho \rangle, Ut\tau \langle \rho \rangle\} \mid t\in T\}=\{\{Ug\langle \rho \rangle, Ug\tau \langle \rho \rangle\} \mid g\in G\}.$$ 
\end{lem}

\begin{proof}
  Recall that $E(\G)=\{\{Uh\langle \rho \rangle, Uh\tau \langle \rho \rangle\} \mid h\in S\}$. Set 
  $$E_T=\{\{Ut\langle \rho \rangle, Ut\tau \langle \rho \rangle\} \mid t\in T\} \hbox{~and~}E_G=\{\{Ug\langle \rho \rangle, Ug\tau \langle \rho \rangle\} \mid g\in G\}.$$ 
  
  For any $h\in S$, there exist $t\in T$ and $g\in G$ such that $Uh=Ut=Ug$. It follows that $Uh\langle \rho \rangle=Ut\langle \rho \rangle=Ug\langle \rho \rangle$. Then
  $$\{Uh\langle \rho \rangle, Uh\tau \langle \rho \rangle\} = \{Ut\langle \rho \rangle, Ut\tau \langle \rho \rangle\}=\{Ug\langle \rho \rangle, Ug\tau \langle \rho \rangle\}.$$
  And so $E(\G)\subseteq E_T$ and $E(\G)\subseteq E_G$. Similarly, we have
  $$E_T\subseteq E_G \hbox{~and~} E_G\subseteq E(\G).$$ Hence $E(\G)=E_T=E_G$.
\end{proof}

Clearly, we have the following proposition about the neighbors of a vertex.
\begin{pro}
Let \( \G = \Mon(G; U, \rho, \tau) \) with $H = \langle \rho \rangle$. Then the neighbourhood of vertex $UkH$ is $\{Ug\tau H \di g\in UkH \hbox{~and~} g\tau\not\in Ug\}$ for $k\in G$.
\end{pro}
\begin{proof}
  Set $N=\{Ug\tau H \di g\in UkH \hbox{~and~} g\tau\not\in Ug\}$. For any $g\in UkH$ such that $g\tau\not\in Ug$, by Proposition \ref{dbc}, we have
  $UgH=UkH$. It follows from  Lemma \ref{epro} that $\{UkH, Ug\tau H\}=\{UgH, Ug\tau H\}\in E(\G)$, and therefore $Ug\tau H$ is a neighbor of $UkH$.
  
  On the other hand, let $UaH$ be a neighbor of $UkH$. By Lemma \ref{epro}, there exists $g\in G$ such that $UkH=UgH$ and $Ug\tau H=UaH$. Then from
Proposition \ref{dbc}, $g\in UkH$. Since $UaH$ is a neighbor of $UkH$, the edge $\{UkH, UaH\}=\{UgH, Ug\tau H\}$ is not a free edge. And so $g\tau\not\in Ug$.

Therefore, the neighbourhood of $UkH$ is $\{Ug\tau H \di g\in UkH \hbox{~and~} g\tau\not\in Ug\}$.
\end{proof}

\begin{lem}
  Let \( \G = \Mon(G; U, \rho, \tau) \). For $h\in G$, the valency of $Uh\langle \rho \rangle$ is $$\frac{|\langle \rho \rangle|}{|\langle \rho \rangle \cap U^h|}.$$
\end{lem}
\begin{proof}
  By Proposition \ref{dbc}, the number of right cosets of $U$ in $G$ that lie in $Uh\langle \rho \rangle$ is $\frac{|\langle \rho \rangle|}{|\langle \rho \rangle \cap U^h|}$. So the valency of $Uh\langle \rho \rangle$ is $\frac{|\langle \rho \rangle|}{|\langle \rho \rangle \cap U^h|}$ from Definition \ref{mong}.
\end{proof}

\begin{theorem}
  Let $\M=(G;\O,\rho,\tau)$ be an algebraic map and $\a\in\O$. Then the underlying graph $\Sigma$ of $\M$ is isomorphic to the monodromy graph $\G=\Mon(G;G_\a,\rho,\tau)$.
\end{theorem}
\begin{proof}
 Set $H= \langle \rho\rangle$. Suppose
\[G = \bigcup_{i=1}^t G_\a g_i H\]
is the double coset decomposition of \( G \) with respect to the subgroups $ G_\a$ and $H$, where \( g_1 \in G_\a \). Then $V(\G)=\{G_\a g_i H \mid i=1,\ldots,t\}$. By Lemma \ref{suborb}, \(\Delta_i = \{\a^g \mid g \in G_\a g_i H\}\), for \( i = 1, \ldots, t \), are all the orbits of \( H\) on \(\Omega\).  Let $$\phi : V(\Sigma) \rightarrow V(\G)$$ be the map defined by $\phi(v) = G_\alpha g_iH$ if and  only if $\text{dart}(v) = \Delta_i$ for all $i \in [t]$.  By Proposition \ref{2.1b}, there is a natural bijection between $V(\Sigma)$ and the cycles of $\rho$
on $\Omega$, with darts $\o, \beta \in \Omega$ lying in the same cycle
if and only if $\text{vert}(\o) = \text{vert}(\beta)$. Since for each $i\in [t]$, $\Delta_i$ forms a cycle of $\rho$, it follows that $\phi$ is a bijection.

By Lemma \ref{2.1a}, there is a natural bijection between $E(\Sigma)$ and the cycles of $\tau$
on $\Omega$, with darts $\o, \beta \in \Omega$ lying in the same cycle
if and only if $\text{edge}(\o) = \text{edge}(\beta)$. For any $\o\in\O$, suppose that $\o=\a^h$ for some $h\in G$. Then $\text{edge}(\o)$ is an edge of $V(\Sigma)$ if and only if $\{G_\a h H, G_\a h\tau H\}$ is an edge of $V(\G)$. It follows that $\phi$ induces a bijection from $E(\Sigma)$ to $E(\G)$. Therefore, $\Sigma$ is isomorphic to $\G$.
\end{proof}

$\G$ is called a \emph{monodromy representation} of $\Sigma$.

Since every connected simple graph  have an embedding on orientable surface, we have the following corollary. Here we give a basic proof.

\begin{cor}\label{mon}
   Every connected simple graph  \( \Sigma \) with more than one vertex is isomorphic to a monodromy graph.
\end{cor}

\begin{proof}

Let \( \Sigma=(V,E,D) \) be the graph with  vertex set \( V \), edge set \( E \), and arc set \( D \).  Let \( \tau \in \text{Sym}(D) \) be an involution such that
\[(u, v)^\tau = (v, u)\]
for all \((u, v) \in D\). Let \(\rho \in \text{Sym}(D)\) satisfy the condition that, for each \(u \in V\), the orbit \((u, v)^{\langle \rho \rangle}\) coincides with the set of all arcs incident to \(u\).

From now on, $\text{ let } \alpha = (u, v) \in D \text{ and } G = \langle \rho, \tau \rangle.$
Setting  \( H = \langle \rho \rangle \), \( K = \langle \tau \rangle \) and \( \G = \Mon(G; G_\alpha, \rho, \tau) \). Clearly, $\Core_G(G_\alpha)=1$. We claim that \( \Sigma \) is isomorphic to \( \G \).

 Set $f(w,z)=w$ for $(w,z)\in D$. Let $$\phi : V \rightarrow V(\G)$$ be the map defined by $\phi(f(\alpha^h)) = G_\alpha hH$ for all $h \in G$. Suppose that $f(\a^h)\ne f(\a^g)$. Then $\a^{G_\alpha hH}\ne \a^{G_\alpha gH}$, and so $G_\alpha hH \ne G_\alpha gH$.
  Note that $G$ is transitive on $D$. So $\phi$ is a bijection.

  Suppose that $\{w,z\}\in E$. Since $G$ is transitive on $D$, there exists an element $b\in G$ such that $\a^b=(w,z)$ and $\a^{b\tau}=(z,w)$. It follows that $\{\phi(f(\alpha^b)),\phi(f(\alpha^{b\tau}))\}=\{G_{\a}bH,G_{\a}b\tau H\}\in E(\G)$.

On the other hand, suppose that $\{G_{\a}hH,G_{\a}g H\}\in E(\G)$. Then from  Lemma \ref{epro}, there exists $x\in G$ such that
$$G_{\a}hH = G_{\a}xH
\hbox{~and~} G_{\a}g H = G_{\a}x\tau K.$$
Then $$f(\a^{h})=f(\a^{x}) \hbox{~and~}=f(\a^g)=f(\a^{x\tau}).$$
It follows that $\{f(\a^{h}), f(\a^{g})\}=\{f(\a^{x}), f(\a^{x\tau})\}\in E$. Hence $\phi$ induces a bijection
from $E$ to $E(\G)$.

 Therefore, \( \Sigma \) is isomorphic to \( \G \).
\end{proof}

Similar to Corollary \ref{mon}, one can get the following theorem.

\begin{theorem}
 Every graph  \( \Sigma \) is isomorphic to a monodromy graph.
\end{theorem}

\section{Automorphisms of monodromy graphs}

In this section, we investigate automorphisms of monodromy graphs.

 Let \( \G = \Mon(G; U, \rho, \tau) \). For $a\in N_G(U)$, define $a_L : V\rightarrow V$ such that $$(Uh\langle \rho \rangle)^{a_L}=a^{-1}Uh\langle \rho \rangle=Ua^{-1}h\langle \rho \rangle$$ for $Uh\langle \rho \rangle\in V$. Set $A=\{a_L \mid a\in N_G(U)\}$.
Then $A$ is called a {\it map automorphism group} of $\G$.

The following theorem is basic.
\begin{theorem}\label{semi}
  Let \( \G = \Mon(G; U, \rho, \tau) \) be a connected monodromy graph and $A=\{a_L \mid a\in N_G(U)\}$.
  Then the following hold.
  \begin{enumerate}[{\rm(i)}]
\item
  $A$ is a subgroup of $\Aut(\G)$.

\item If $U=1$, then $\G$ is arc-transitive.
  \end{enumerate}
\end{theorem}

\begin{proof} Let $H=\langle \rho \rangle$.
  \begin{enumerate}[{\rm(i)}]
\item Let $S$ be a set of right
coset representatives for $U$ in $G$. Take $a\in N_G(U)$  and edge $\{UhH, Uh\tau H\}$ with $h\in S$. Then $a_L \in A$, and we have
\[ \{UhH, Uh\tau H\}^{a_L} = \{Ua^{-1}hH, Ua^{-1}h\tau H\} \]
which is also an edge in $\G$ from Lemma \ref{epro}. Let $D=\{g\in S \mid g\in Uh\langle \rho \rangle \hbox{~and~} g\tau \subseteq Uh\tau \langle \rho \rangle\}$, $Q=a^{-1}D=\{a^{-1}g \mid g\in D\}$ and  $D_a=\{g\in S \mid g\in Ua^{-1}h\langle \rho \rangle \hbox{~and~} g\tau \subseteq Ua^{-1}h\tau \langle \rho \rangle\}$. Then for all $x,y\in Q$, $Ux\ne Uy$ if and only if $x\ne y$. Note that $Ua^{-1}=a^{-1}U$, one can get $|D|=|Q|\le D_a$. Similarly, we have $|D_a|\le |D|$. And so
$$m(Uh\langle \rho \rangle, Uh\tau \langle \rho \rangle)=m(Ua^{-1}hH, Ua^{-1}h\tau H).$$
 Thus $a_L \leq \Aut(\G)$. For any $b_L \in A$, we have
\begin{align*}
\{UhH, Uh\tau H\}^{(ab^{-1})_L}
&= \{U(ab^{-1})^{-1}hH, U(ab^{-1})^{-1}h\tau H\} \\
&= \{Uba^{-1}hH, Uba^{-1}h\tau H\} \\
&= \left(\{UhH, Uh\tau H\}^{a_L}\right)^{b^{-1}_L}.
\end{align*}
Hence $A \leq \Aut(\G)$.
\item When $U=1$, $A=G_L$ acts transitively on the vertices (as $V = \{H, gH, \ldots\}$) and the stabilizer of $H$ is $H$ itself, which acts transitively on its neighbors, making $\G$ arc-transitive.
\end{enumerate}
\end{proof}

From Theorem \ref{semi}, we have the following corollary.
\begin{cor}\label{arc}

  Let $\G = \mathrm{Mon}(G; U, \rho, \tau)$ be a monodromy representation of a graph. Then
  $\Sigma=\mathrm{Mon}(G; 1, \rho, \tau)$ is an arc-transitive graph.
\end{cor}
%\begin{proof}
%By Theorem \ref{semi}, $\Sigma$ is an arc-transitive graph. Since $\G$ is simple, one can see that $\Sigma$ is simple from Proposition \ref{simple}. 
%\end{proof}

\begin{rem}
Corollary \ref{arc} shows that
every graph gives rise to an arc-transitive graph through its monodromy representation.
\end{rem}

\section{Enumerating maps with a given monodromy group}

Based on Proposition \ref{iso}, we denote an algebaric map  $\mathcal{M}(G;\O,\rho,\tau)$ by $\mathcal{M}(G;U,\rho,\tau)$ where $U$ is a stabiliser in $G$. 

\begin{dnt}
  Let $G=\langle \rho, \tau \rangle$ with $\tau^2=1$ be a group. Suppose that $U$ is a subgroup of $G$ such that
  the permutation representation $\pi$ of $G$ on the right cosets of $U$ is faithful. Then
  $\M(G; U, \rho,\tau)$ is called an algebaric map.
\end{dnt}

\begin{theorem}
  Let $\M=\M(G; U, \rho,\tau)$ and $\widetilde{\M}=\M(\widetilde{G}; \widetilde{U}, \widetilde{\rho},\widetilde{\tau})$ be two maps.
  Then $\M\cong \widetilde{\M}$ if and only if there exists a group isomorphism $\sigma$ from $G$ to $\widetilde{G}$ such that for some $ \widetilde{y}\in  \widetilde{G}$,
  $$U^\sigma= \widetilde{y}^{-1}\widetilde{U}\widetilde{y},~ \rho^\sigma = \widetilde{\rho} \hbox{~and~} \tau^\sigma= \widetilde{\tau}.$$
\end{theorem}

\begin{proof}
Consider the permutation representations $\pi$ and $\widetilde{\pi}$ induced by the actions of $G$ on $[G: U]$ and $\widetilde{G}$ on $[\widetilde{G}:\widetilde{U}]$, respectively. We denote $Uxa$ by $(Ux)^a$, that is, $(Ux)^a := Uxa$, for $x,a\in G$.

Suppose that there exists a group isomorphism $\sigma$ from $G$ to $\widetilde{G}$ such that for some $ \widetilde{y}\in  \widetilde{G}$,
  $$U^\sigma= \widetilde{y}^{-1}\widetilde{U}\widetilde{y},~ \rho^\sigma = \widetilde{\rho} \hbox{~and~} \tau^\sigma = \widetilde{\tau}.$$
Setting $\mu: Ux \mapsto \widetilde{U}\widetilde{y}^{-1}x^\sigma$. Then $\mu$ is a bijection from $[G:U]$ to $[\widetilde{G}:\widetilde{U}]$.
It follows that for all $x\in G$, $$(Ux)^{\mu\widetilde{\rho}}=(Ux)^{\rho\mu} \hbox{~and~} (Ux)^{\mu\widetilde{\tau}}=(Ux)^{\tau\mu}.$$
Hence $\M$ is isomorphic to $\widetilde{\M}$.

  Conversely, suppose that $\M$ is isomorphic to $\widetilde{\M}$. Then there exists a bijection $\mu$ from $[G: U]$ to $[\widetilde{G}:\widetilde{U}]$ such that for all $x\in G$, $$(Ux)^{\mu\widetilde{\rho}}=(Ux)^{\rho\mu} \hbox{~and~} (Ux)^{\mu\widetilde{\tau}}=(Ux)^{\tau\mu}.$$
  We can define a map $\sigma$ from $G$ to $\widetilde{G}$ by setting $\widetilde{g}=g^{\sigma}$ if and only if $(Ux)^{\mu\widetilde{g}}=(Ux)^{g\mu}$ for all $x\in G$. We claim that $\sigma$ is a group isomorphism. The faithfulness of the permutation representations $\pi$ and $\widetilde{\pi}$ implies that $\sigma$ is bijective.
  It is easily verified that $\sigma$ preserves the group operation, hence $\sigma$ is a group isomorphism. It follows that $\pi$ and $\widetilde{\pi}$ are permutation isomorphic.

Now assume that $\mu$ maps $U$ to $\widetilde{U}\widetilde{y}$. Then for the permutation representation $\pi$, the stabilizer of the point $U$ is $U$; and for the permutation representation $\widetilde{\pi}$, the stabilizer of the point $\widetilde{U}$ is $\widetilde{y}^{-1}\widetilde{U}\widetilde{y}$.
Since $\pi$ and $\widetilde{\pi}$ are permutation isomorphic, we must have $U^{\sigma} = \widetilde{y}^{-1}\widetilde{U}\widetilde{y}$, from which the desired conclusion follows.
\end{proof}

\begin{cor}
  Let $\M=\M(G; U, \rho,\tau)$ and $\widehat{\M}=\M(G; W, \hat{\rho},\hat{\tau})$ be two maps.
  Then $\M\cong \widehat{\M}$ if and only if there exists an automorphism $\sigma\in \Aut(G)$ such that for some $y\in  G$,
  $$U^\sigma= y^{-1}Wy,~ \rho^\sigma = \hat{\rho} \hbox{~and~} \tau^\sigma= \hat{\tau}.$$
\end{cor}

\begin{cor}
  Let $\M=\M(G; U, \rho,\tau)$ and $\widehat{\M}=\M(G; W, \rho,\tau)$ be two maps.
  Then $\M\cong \widehat{\M}$ if and only if $U$ and $W$ are conjugate in $G$.
\end{cor}

\begin{theorem}\label{count}
  Let $m$ be the number of conjugacy classes of core-free subgroups of $G$. Suppose that $G$ admits $n$ pairwise non-isomorphic regular maps. Then the number of non-isomorphic maps admitted by $G$ is $mn$.
\end{theorem}

\section{Platonic solids arising from planar graphs}

Let $G$ be a group, and $\rho_i,\tau_i\in G$ for $i\in \{1,2\}$. The pair $(\rho_1,\tau_1)$ is said to be {\it isomorphic} to $(\rho_2,\tau_2)$ if there 
exists an automorphism $\sigma$ of $G$ such that $\rho_1^\sigma=\rho_2$ and $\tau_1^\sigma=\tau_2$.

With the help of  {\sc Magma}\cite{BCP}, one can get the following three lemmas.

\begin{lem}\label{A5}
Suppose that $\Mon(A_5;U, \rho,\tau)$ is a planar graph with the monodromy group $A_5$. Then $U$ is conjugate to one of the following groups.
  \begin{enumerate}[{\rm(i)}]
\item
  $U_1=\langle (1) \rangle\cong C_1$.
\item $U_2=\langle  (1, 2)(4, 5) \rangle\cong C_2$.
\item $U_3=\langle   (1, 5, 4) \rangle\cong C_3$.
\item $U_4=\langle   (1, 2)(4, 5)
    (1, 5)(2, 4) \rangle\cong C_2\times C_2$.
\item $U_5=\langle   (1, 5, 2, 4, 3) \rangle\cong C_5$.
\item $U_6=\langle   (1, 5)(3, 4),
    (1, 2, 5) \rangle\cong S_3$.
\item $U_7=\langle   (1, 3)(4, 5),
    (1, 5, 2, 4, 3) \rangle\cong D_{10}$.
\item $U_8=\langle    (1, 5, 4),
    (1, 2)(4, 5),
    (1, 5)(2, 4) \rangle\cong A_4$.

  \end{enumerate}
And $(\rho, \tau)$ is isomorphic to one of the following pairs.
\begin{enumerate}[{\rm(i)}]
\item
  $(\rho_1, \tau_1)=( (1, 5, 4, 3, 2),
        (1, 2)(3, 4))$.
\item $(\rho_2, \tau_2)=((1, 5, 4),
        (1, 2)(3, 4))$.
\item $(\rho_3, \tau_3)=((1, 4, 2, 5, 3),
        (1, 2)(3, 4))$.
  \end{enumerate}
Moreover,  for each $i\in [8]$ and $j\in [2]$, $\Mon(A_5;U_i, \rho_j,\tau_j)$, $\Mon(A_5;U_5, \rho_3,\tau_3)$ and $\Mon(A_5;U_7, \rho_3,\tau_3)$ are planar graphs.
\end{lem}

\begin{lem}\label{A4}
Suppose that $\Mon(A_4; U, \rho,\tau)$ is a planar graph  with the monodromy group $A_4$. Then $U$ is conjugate to one of the following groups.
  \begin{enumerate}[{\rm(i)}]
\item
  $U_1=\langle (1) \rangle\cong C_1$.
\item
  $U_2=\langle (1, 3)(2, 4) \rangle\cong C_2$.
\item
  $U_3=\langle  (2, 3, 4) \rangle\cong C_3$.

  \end{enumerate}
  And $(\rho, \tau)$ is isomorphic to $(\rho_1, \tau_1)=( (1, 3, 2),
        (1, 2)(3, 4))$.
Moreover, for each $i\in [3]$, $\Mon(A_4;U_i, \rho_1,\tau_1)$ is a planar graph. 
\end{lem}

\begin{lem}\label{S4}
Suppose that $\Mon(S_4; U, \rho,\tau)$ is a planar graph  with the monodromy group $S_4$. Then $U$ is conjugate to one of the following groups.
  \begin{enumerate}[{\rm(i)}]
\item
  $U_1=\langle (1) \rangle\cong C_1$.
\item
  $U_2=\langle (1, 4)(2, 3) \rangle\cong C_2$.
\item
  $U_3=\langle  (3, 4) \rangle\cong C_2$.
\item
  $U_4=\langle  (2, 3, 4) \rangle\cong C_3$.
\item
  $U_5=\langle   (1, 4, 2, 3) \rangle\cong C_4$.
\item
  $U_6=\langle   (3, 4),
    (1, 2)(3, 4) \rangle\cong C_2\times C_2$.
\item
  $U_7=\langle  (3, 4),
    (2, 3, 4) \rangle\cong S_3$.

  \end{enumerate}
And $(\rho, \tau)$ is isomorphic to one of the following pairs.
\begin{enumerate}[{\rm(i)}]
\item
  $(\rho_1, \tau_1)=(  (1, 2, 3, 4),
        (2, 3))$.
\item $(\rho_2, \tau_2)=( (1, 2, 4),
        (2, 3))$.
  \end{enumerate}
Moreover, for each $i\in [7]$ and $j\in [2]$, $\Mon(S_4;U_i, \rho_j,\tau_j)$ is a planar graph. 
\end{lem}

\vskip 0.5cm
In the following ten figures, we show that where there is a graph, there is a group, an arc-transitive graph, and an orientable regular map. That is, Figure $i+1$ can be obtained from Figure $i$ for $i\in \{1,3,5,7,9\}$.

\vskip 1cm
\begin{figure}[H]
\centering
\resizebox{0.7\textwidth}{!}{%
\begin{circuitikz}
\tikzstyle{every node}=[font=\LARGE]
\draw  (9.25,13.5) arc (180:0:2cm and 1.5cm);
\draw  (6.75,13.5) arc (180:360:3.25cm and 2cm);
\node at (9.25,13.5) [circ] {};
\node at (6.75,13.5) [circ] {};
\draw  (10.5,13.5) ellipse (1.25cm and 0.75cm);
\draw  (5.5,13.5) ellipse (1.25cm and 0.75cm);
\node at (10.75,13.5) [circ] {};
\node at (5.25,13.5) [circ] {};
\draw (5.25,13.5) to[short] (6.75,13.5);
\draw (6.75,13.5) to[short] (9.25,13.5);
\draw (9.25,13.5) to[short] (10.75,13.5);
\end{circuitikz}
}%
\caption{The monodromy graph Mon$(A_5;U_5,\rho_1,\tau_1)$ in Lemma \ref{A5}.}
\label{fig:my_label}
\end{figure}
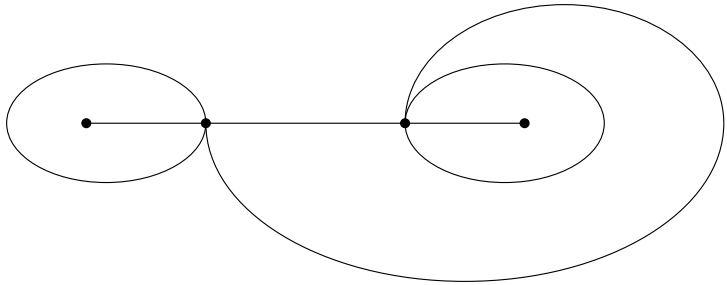

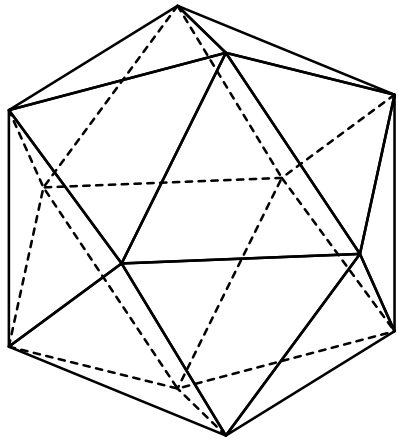
\begin{figure}[H]

\centering
\scalebox{0.8}{
\begin{tikzpicture}[scale=2, line cap=round,line join=round,>=latex,x=1.0cm,y=1.0cm]

\coordinate (P0) at (0.19866933079506122, 1.3910718630305983);
\coordinate (P1) at (0.19866933079506122, -1.7804902053392486);
\coordinate (P2) at (-0.19866933079506122, 1.7804902053392486);
\coordinate (P3) at (-0.19866933079506122, -1.3910718630305983);
\coordinate (P4) at (1.301520307589847, -0.27557655385046137);
\coordinate (P5) at (0.6586128480926363, 0.35451555984757627);
\coordinate (P6) at (-0.6586128480926363, -0.35451555984757627);
\coordinate (P7) at (-1.301520307589847, 0.27557655385046137);
\coordinate (P8) at (1.5857810341849234, 1.0439295752119735);
\coordinate (P9) at (-1.5857810341849234, 0.9162035804705098);
\coordinate (P10) at (1.5857810341849234, -0.9162035804705098);
\coordinate (P11) at (-1.5857810341849234, -1.0439295752119735);

% Background varying in size depending on the main figure
\fill[color=white] (-2,-2) rectangle (2,2);

% 实线面（只画边框，不填充）
\draw[very thick, black] (P0) -- (P2) -- (P8) -- cycle;
\draw[very thick, black] (P0) -- (P2) -- (P9) -- cycle;
\draw[very thick, black] (P0) -- (P8) -- (P4) -- cycle;
\draw[very thick, black] (P0) -- (P9) -- (P6) -- cycle;
\draw[very thick, black] (P0) -- (P4) -- (P6) -- cycle;

\draw[very thick, black] (P9) -- (P11) -- (P6) -- cycle;
\draw[very thick, black] (P1) -- (P11) -- (P6) -- cycle;
\draw[very thick, black] (P1) -- (P4) -- (P6) -- cycle;

\draw[very thick, black] (P4) -- (P10) -- (P8) -- cycle;
\draw[very thick, black] (P4) -- (P10) -- (P1) -- cycle;

% 虚线边
\draw[very thick, dashed, black] (P1) -- (P3);
\draw[very thick, dashed, black] (P11) -- (P3) -- (P10);
\draw[very thick, dashed, black] (P11) -- (P7) -- (P9);
\draw[very thick, dashed, black] (P2) -- (P7) -- (P3);
\draw[very thick, dashed, black] (P5) -- (P7);
\draw[very thick, dashed, black] (P2) -- (P5) -- (P3);
\draw[very thick, dashed, black] (P8) -- (P5) -- (P10);

\end{tikzpicture} }

\caption{A regular icosahedron $\M(A_5;U_1,\rho_1,\tau_1)$ with the underlying graph Mon$(A_5;U_1,\rho_1,\tau_1)$ in Lemma \ref{A5}.}
\label{fig:ico}
\end{figure}

\begin{figure}[H]
\centering
\scalebox{0.95}{
\resizebox{0.8\textwidth}{!}{%
\begin{circuitikz}
\tikzstyle{every node}=[font=\LARGE]
\node at (10,12.25) [circ] {};
\node at (12.5,12.25) [circ] {};
\node at (15,12.25) [circ] {};
\node at (7.5,12.25) [circ] {};
\draw  (6.75,12.25) ellipse (0.75cm and 0.5cm);
\draw  (15.75,12.25) ellipse (0.75cm and 0.5cm);
\draw (7.5,12.25) to[short] (15,12.25);
\draw  (10,12.25) arc (0:180:2.5cm and 1.5cm);
\draw  (5,12.25) arc (180:360:3.75cm and 2cm);
\end{circuitikz}
}%
}
\caption{The monodromy graph Mon$(A_5;U_5,\rho_2,\tau_2)$ in Lemma \ref{A5}.}
\label{fig:my_label}
\end{figure}

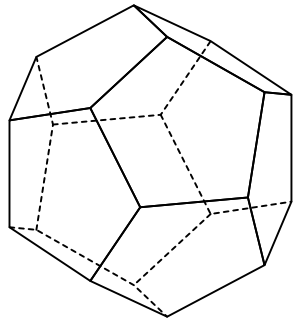
\begin{figure}[H]
\centering
\scalebox{0.6}{
\begin{tikzpicture}[scale=2, line cap=round,line join=round,>=latex,x=1.0cm,y=1.0cm]

\coordinate (P0) at (1.2508566957869456, 0.7603474449732492);
\coordinate (P1) at (1.2508566957869456, -1.1503255332779627);
\coordinate (P2) at (0.6598162824642664, 1.3249899183682845);
\coordinate (P3) at (0.6598162824642664, -0.5856830598829275);
\coordinate (P4) at (-0.6598162824642664, 0.5856830598829275);
\coordinate (P5) at (-0.6598162824642664, -1.3249899183682845);
\coordinate (P6) at (-1.2508566957869456, 1.1503255332779627);
\coordinate (P7) at (-1.2508566957869456, -0.7603474449732492);
\coordinate (P8) at (0.18264153207910092, 1.3712827900732547);
\coordinate (P9) at (0.18264153207910092, -1.7202510301231948);
\coordinate (P10) at (-0.18264153207910092, 1.7202510301231948);
\coordinate (P11) at (-0.18264153207910092, -1.3712827900732547);
\coordinate (P12) at (1.0685921597130592, -0.40283109341752804);
\coordinate (P13) at (0.11226868223217829, 0.5107796200274473);
\coordinate (P14) at (-0.11226868223217829, -0.5107796200274473);
\coordinate (P15) at (-1.0685921597130592, 0.40283109341752804);
\coordinate (P16) at (1.5457669100982248, 0.7317368768227392);
\coordinate (P17) at (-1.5457669100982248, 0.44912396512249836);
\coordinate (P18) at (1.5457669100982248, -0.44912396512249836);
\coordinate (P19) at (-1.5457669100982248, -0.7317368768227392);

% Background varying in size depending on the main figure
\fill[color=white] (-2,-2) rectangle (2,2);

% 实线面（只画边框，不填充）
\draw[very thick, black] (P0) -- (P8) -- (P10) -- (P2) -- (P16) -- cycle;

\draw[very thick, black] (P4) -- (P8) -- (P10) -- (P6) -- (P17) -- cycle;

\draw[very thick, black] (P0) -- (P8) -- (P4) -- (P14) -- (P12) -- cycle;

\draw[very thick, black] (P16) -- (P0) -- (P12) -- (P1) -- (P18) -- cycle;

\draw[very thick, black] (P1) -- (P12) -- (P14) -- (P5) -- (P9) -- cycle;

\draw[very thick, black] (P4) -- (P14) -- (P5) -- (P19) -- (P17) -- cycle;

% 虚线边
\draw[very thick, dashed, black] (P19) -- (P7) -- (P11) -- (P9);
\draw[very thick, dashed, black] (P7) -- (P15) -- (P6);
\draw[very thick, dashed, black] (P15) -- (P13) -- (P2);
\draw[very thick, dashed, black] (P13) -- (P3) -- (P18);
\draw[very thick, dashed, black] (P3) -- (P11);

\end{tikzpicture} }
\caption{A regular dodecahedron $\M(A_5;U_1,\rho_2,\tau_2)$ with the underlying graph Mon$(A_5;U_1,\rho_2,\tau_2)$ in Lemma \ref{A5}.}
\label{fig:dod}
\end{figure}

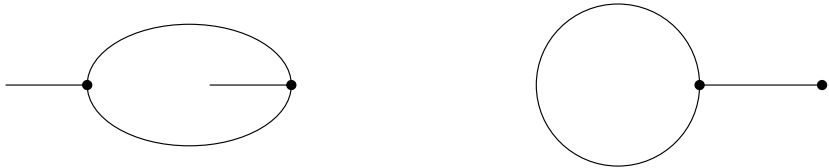
\begin{figure}[H]
\centering
\resizebox{0.8\textwidth}{!}{%
\begin{circuitikz}
\tikzstyle{every node}=[font=\LARGE]
\draw  (12.75,13.5) circle (1cm);
\draw  (7.5,13.5) ellipse (1.25cm and 0.75cm);
\node at (13.75,13.5) [circ] {};
\node at (8.75,13.5) [circ] {};
\node at (6.25,13.5) [circ] {};
\draw (7.75,13.5) to[short] (8.75,13.5);
\draw (5.25,13.5) to[short] (6.25,13.5);
\draw (13.75,13.5) to[short] (15.25,13.5);
\node at (15.25,13.5) [circ] {};
\end{circuitikz}
}%
\caption{The monodromy graphs Mon$(A_4;U_2,\rho_1,\tau_1)$ (left) with two free edges and Mon$(A_4;U_3,\rho_1,\tau_1)$ (right) in Lemma \ref{A4}.}
\label{fig:my_label}
\end{figure}

\begin{figure}[H]
\centering
\scalebox{0.8}{
\begin{tikzpicture}[scale=2, line cap=round,line join=round,>=latex,x=1.0cm,y=1.0cm]

\coordinate (P0) at (-1.1042758977624227, 1.3096936302264104);
\coordinate (P1) at (0.8835014100843277, -0.3031514993551603);
\coordinate (P2) at (-0.8835014100843277, -1.4214862452202075);
\coordinate (P3) at (1.1042758977624227, 0.41494411434895745);

% Background varying in size depending on the main figure
\fill[color=white] (-2,-2) rectangle (2,2);

% 实线面（只画边框，不填充）
\draw[very thick, black] (P0) -- (P1) -- (P2) -- cycle;

\draw[very thick, black] (P0) -- (P3) -- (P1) -- cycle;

% 虚线边
\draw[very thick, dashed, black] (P2) -- (P3);

\end{tikzpicture} }

\caption{A regular tetrahedron $\M(A_4;U_1,\rho_1,\tau_1)$ with the underlying graph Mon$(A_4;U_1,\rho_1,\tau_1)$ in Lemma \ref{A4}.}
\label{fig:tet}
\end{figure}

\begin{figure}[H]
\centering
\scalebox{0.45}{
\resizebox{0.8\textwidth}{!}{%
\begin{circuitikz}
\tikzstyle{every node}=[font=\LARGE]
\draw  (10,13.5) circle (1.75cm);
\node at (8.25,13.5) [circ] {};
\node at (10,15.25) [circ] {};
\node at (10,11.75) [circ] {};
\node at (10.75,13.5) [circ] {};
\draw (10,15.25) to[short] (10,11.75);
\draw [short] (10,15.25) -- (10.75,13.5);
\draw [short] (10.75,13.5) -- (10,11.75);
\end{circuitikz}
}%
}
\caption{The monodromy graph Mon$(S_4;U_2,\rho_1,\tau_1)$ in Lemma \ref{S4}.}
\label{fig:my_label}
\end{figure}
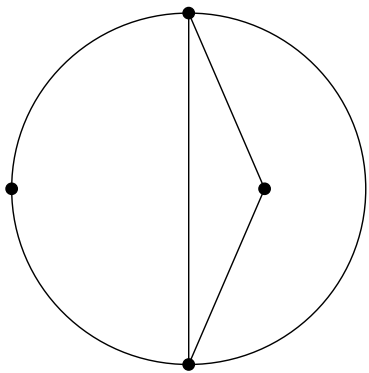

\begin{figure}[H]
\centering
\begin{tikzpicture}[scale=2, line cap=round,line join=round,>=latex,x=1.0cm,y=1.0cm]

\coordinate (P0) at (-0.27516333805159693, -0.21095489306059886);
\coordinate (P1) at (0.27516333805159693, 0.21095489306059886);
\coordinate (P2) at (-0.9613974918795568, 0.06037778654840108);
\coordinate (P3) at (0.9613974918795568, -0.06037778654840108);
\coordinate (P4) at (-0.0, -0.9756293127952373);
\coordinate (P5) at (-0.0, 0.9756293127952373);

% Background varying in size depending on the main figure
\fill[color=white] (-1.2,-1.2) rectangle (1.2,1.2);

% 实线面（只画边框，不填充）
\draw[very thick, black] (P4) -- (P2) -- (P0) -- cycle;

\draw[very thick, black] (P4) -- (P3) -- (P0) -- cycle;

\draw[very thick, black] (P5) -- (P2) -- (P0) -- cycle;

\draw[very thick, black] (P5) -- (P3) -- (P0) -- cycle;

% 虚线边
\draw[very thick, dashed, black] (P1) -- (P2);
\draw[very thick, dashed, black] (P1) -- (P3);
\draw[very thick, dashed, black] (P1) -- (P4);
\draw[very thick, dashed, black] (P1) -- (P5);

\end{tikzpicture}

\caption{A regular octahedron $\M(S_4;U_1,\rho_1,\tau_1)$ with the underlying graph Mon$(S_4;U_1,\rho_1,\tau_1)$ in Lemma \ref{S4}.}
\label{fig:oct}
\end{figure}
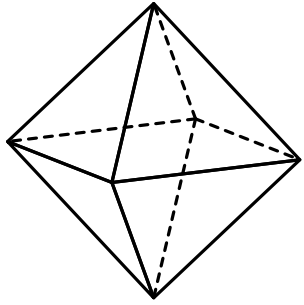

\begin{figure}[H]
\centering
\scalebox{0.45}{
\resizebox{0.8\textwidth}{!}{%
\begin{circuitikz}
\tikzstyle{every node}=[font=\LARGE]
\draw  (7.5,13.5) ellipse (0.5cm and 1.25cm);
\draw  (10,13.5) ellipse (0.5cm and 1.25cm);
\node at (7.5,14.75) [circ] {};
\node at (10,14.75) [circ] {};
\node at (10,12.25) [circ] {};
\node at (7.5,12.25) [circ] {};
\draw (7.5,14.75) to[short] (10,14.75);
\draw (7.5,12.25) to[short] (10,12.25);
\end{circuitikz}
}%

}
\caption{The monodromy graph Mon$(S_4;U_2,\rho_2,\tau_2)$ in Lemma \ref{S4}.}
\label{fig:my_label}
\end{figure}
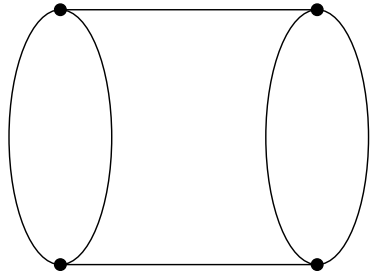

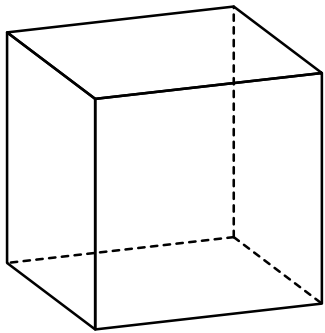
\begin{figure}[H]
\centering
\scalebox{0.8}{
\begin{tikzpicture}[scale=2, line cap=round, line join=round, >=latex, x=1.0cm, y=1.0cm]

\coordinate (P0) at (-0.5696813314905527, 0.5728160673358398);
\coordinate (P1) at (-0.5696813314905527, -1.3378569109153722);
\coordinate (P2) at (1.2943968404439, 0.7869841443124108);
\coordinate (P3) at (1.2943968404439, -1.1236888339388011);
\coordinate (P4) at (-1.2943968404439, 1.1236888339388011);
\coordinate (P5) at (-1.2943968404439, -0.7869841443124108);
\coordinate (P6) at (0.5696813314905527, 1.3378569109153722);
\coordinate (P7) at (0.5696813314905527, -0.5728160673358398);

% 背景
\fill[color=white] (-2,-2) rectangle (2,2);

% 三个面的连线（实线）
\draw[very thick] (P0) -- (P1) -- (P3) -- (P2) -- cycle;
\draw[very thick] (P0) -- (P1) -- (P5) -- (P4) -- cycle;
\draw[very thick] (P0) -- (P2) -- (P6) -- (P4) -- cycle;

% 虚线边
\draw[very thick, dashed] (P7) -- (P6);
\draw[very thick, dashed] (P7) -- (P5);
\draw[very thick, dashed] (P7) -- (P3);

\end{tikzpicture} }
\caption{A regular cube $\M(S_4;U_1,\rho_2,\tau_2)$ with the underlying graph Mon$(S_4;U_1,\rho_2,\tau_2)$ in Lemma \ref{S4}.}
\label{fig:cub}
\end{figure}

Based on Theorem \ref{count} and Lemma \ref{A5},\ref{A4}, \ref{S4}, we can get the following theorem.

\begin{theorem}
  There are $24$, $3$, and $14$ non-isomorphic maps with monodromy groups $A_5$, $A_4$, and $S_4$, respectively.
\end{theorem}

\section*{Data Availability}
 No data was used in the preparation of this manuscript.

\section*{Conflict of interest}
The authors declared no potential conflicts of interest with respect to the research, authorship, and publication of this article.

\end{document}